\documentclass[11pt]{article}
\usepackage{amsmath, amssymb, amsfonts, amsthm, latexsym, hyperref, graphicx}
\usepackage{pstricks}
\usepackage{enumerate}

\newtheorem{thm}{Theorem}

\newtheorem{obs}{Observation}

\newtheorem{propos}{Proposition}

\newtheorem{problem}{Problem}

\newcommand{\degr}{\text{deg}}

\newcommand{\QR}{\text{QR}}

\newcommand{\C}{\mathbb{C}}

\newcommand{\T}{\mathbb{T}}
\newcommand{\ty}{\text{Ty}}

\title{Drawing outerplanar graphs}
\author{Noga Alon \thanks{Sackler School of Mathematics and
Blavatnik School of Computer Science,
Tel Aviv University,
Tel Aviv 69978, Israel. Email: nogaa@tau.ac.il.
Research supported in part by an ERC Advanced
grant,
by a USA-Israeli BSF grant, and by the Hermann Minkowski Minerva
Center for Geometry
in Tel Aviv University.}
\and Ohad N. Feldheim \thanks{Sackler School of Mathematics,
Tel Aviv University,
Tel Aviv 69978, Israel. Email: ohad\_f@netvision.net.il.
Research supported in part by an ERC Advanced
grant.  } }

\begin{document}
\maketitle

\begin{abstract}
It is shown that for any outerplanar graph G there is a one to one mapping of
the vertices of G to the plane, so that the number of distinct distances
between pairs of connected vertices is at most three. This settles a problem
of Carmi, Dujmovic, Morin and Wood. The proof combines (elementary)
geometric, combinatorial, algebraic and probabilistic arguments.
\end{abstract}

\section{Introduction}
A \emph{linear} embedding of a graph $G$ is a mapping of the vertices of $G$
to distinct points in the plane. The image of every edge $uv$ of the graph is
the open interval between the image of $u$ and the image of $v$. The length
of that interval is called the \emph{edge-length} of $uv$ in the embedding. A
\emph{degenerate drawing} of a graph $G$ is a linear embedding in which the
images of all vertices are distinct. A \emph{drawing} of $G$ is a degenerate
drawing in which the image of every edge is disjoint from the image of every
vertex. The \emph{distance-number} of a graph is the minimum number of
distinct edge-lengths in a drawing of $G$, the \emph{degenerate
distance-number} is its counterpart for degenerate drawings.

An \emph{outerplanar} graph is a graph that can be embedded in the plane
without crossings in such a way that all the vertices lie in the boundary of
the unbounded face of the embedding. In \cite{Pvpd}, Carmi, Dujmovic, Morin
and Wood ask if the degenerate distance-number of outerplanar graphs are
uniformly bounded. We answer this positively by showing that the degenerate
distance number of outerplanar graphs is at most
$3$. 
This result is derived by explicitly constructing a degenerate drawing for
every such graph.

\begin{thm}\label{thm: 3 degenerate}
For almost every triple $a, b, c\in (0,1)$, every outerplanar graph has a
degenerate drawing using only edge-lengths $a,b$ and $c$.
\end{thm}

For matters of convenience, throughout the paper we consider all linear
embeddings as mapping vertices to the complex plane.

\section{Background and Motivation}

While the distance-number and the degenerate distance-number of a graph are
two natural notions in the context of representing a graph as a diagram in
the plane, this was not the sole motivation to their introduction.

Both notions were introduced by Carmi, Dujmovic, Morin and Wood in
\cite{Pvpd}, and generalize several well studied problems. Indeed,
Erd\H{o}s suggested in \cite{Erd} the problem of determining or
estimating  the
minimum possible number
of distinct distances between $n$ points in the plane. This problem can be
rephrased as finding the degenerate distance-number of $K_n$, the complete
graph on $n$ vertices. Recently, Guth and Katz, in a
ground-breaking paper \cite{GK}, established a
lower-bound of $cn / \log n$ on this number, which almost matches the
$O(n / \sqrt{\log n}) $
upper-bound due to Erd\H{o}s. Another problem,
considered by Szemer\'edi (See Theorem 13.7 in \cite{Smz}), is that of
finding the minimum possible number of distances between $n$ non-collinear
points in the plane. This problem can be rephrased as finding the
distance-number of $K_n$. One interesting consequence of the known results on
these questions is that the distance-number and
the degenerate distance-number of
$K_n$ are not the same, thus justifying the two separate notions. For a short
survey of the history of both problems, including some classical bounds, the
reader is referred to the background section of \cite{Pvpd}.

Another notion which is generalized by the degenerate distance-number is that
of a unit-distance graph, that is, a graph that can be embedded in the plane
so that two vertices are at distance one if and only if they are connected by
an edge. Observe that all unit-distance graphs have degenerate
distance-number $1$ while the converse is not true. Constructing "dense"
unit-distance graphs is a classical problem.  The best construction, due to
Erd\H{o}s \cite{Erd}, gives an $n$-vertex unit-distance graph with
$n^{1+c/\log \log n}$ edges, while the best known upper-bound, due to
Spencer, Szemer\'edi and Trotter \cite{SpSz}, is $cn^{4/3}$ (A simpler proof
for this bound was found by Sz\'ekely, see \cite{Szk}). Note that this
implies that the $k$ most frequent interpoint distances between $n$ points
occur in total no more than $ckn^{4/3}$ times, and thus that a graph with
degenerate distance-number $k$ may have no more than $ckn^{4/3}$ edges. Katz
and Tardos gave in \cite{KT} another bound on the frequency of interpoint
distances between $n$ points in the plane, which yields that a graph with
distance-number $k$ may have no more than $cn^{1.46}k^{0.63}$ edges.

After introducing the notions of distance-number and degenerate
distance-number, Carmi, Dujmovic, Morin and Wood
studied in \cite{Pvpd} the
behavior of bounded degree graphs with respect to these notions. They
show that graphs with bounded degree greater or equal to five can have
degenerate distance-number arbitrarily large, giving a polynomial lower-bound
for graphs with bounded degree greater or equal to seven. They also give a $c
\log(n)$ upper-bound to the distance-number of bounded degree graphs with
bounded treewidth. In the same paper, the authors ask whether this bound can
be improved for outerplanar graphs, and in particular whether
such graphs have a uniformly bounded 
degenerate distance-number, a question which we answer here positively.

\section{Preliminaries}
%

{\bf Outerplanarity, $\Delta$-trees and $T^*$.} An \emph{outerplanar} graph
is a graph that can be embedded in the plane without crossings so that all
its vertices lie in the boundary of the unbounded face of the embedding. The
edges which border this unbounded face are uniquely defined, and are called
the \emph{external} edges of the graph; the rest of the edges are called
\emph{internal}.

Let $\Delta$ be the triangle graph, that is, a graph on three vertices $v_0$,
$v_1$, and $v_2$, whose edges are $v_0v_1$, $v_0v_2$ and $v_2v_1$. A graph is
said to be a \emph{$\Delta$-tree} if it can be generated from $\Delta$ by
iterations of adding a new vertex and connecting it to both ends of some
external edge other than $v_0v_1$. This results in an outerplanar graph whose
bounded faces are all triangles. The adjacency graph of the bounded faces of
such a graph is a binary tree, that is -- a rooted tree of maximal degree 3.
In fact, all $\Delta$-trees are subgraphs of an infinite graph $T^*$. All
bounded faces of $T^*$ are triangles, and the adjacency graph of those faces
is a complete infinite binary tree. The root of $T^*$ is denoted by
$T^*_{\text{root}}$. An illustration of a $\Delta$-tree can be found in the
left hand side of figure~\ref{fig: T star}.

It is a known fact, which can be proved using induction, that the triangulation of every
outerplanar graph is a $\Delta$-tree. All outerplanar graphs are therefore
subgraphs of $T^*$, a fact which reduces Theorem~\ref{thm: 3 degenerate} to the
following:
\begin{propos}\label{prop: 3 degenerate}
For almost every triple $a, b, c\in (0,1)$, the graph $T^*$ has a degenerate
drawing using only edge-lengths $a$, $b$ and $c$.
\end{propos}

{\bf The rhombus graph $H$, Covering $T^*$ by rhombi.} In order to prove the
above proposition, we construct an explicit embedding of $T^*$ in $\C$. To do
so we introduce a covering of $T^*$ by copies of a particular directed graph
$H$ which we call a \emph{rhombus}. We then embed $T^*$ into $\C$, one copy
of $H$ at a time.

  \begin{figure}[htb!]
   \centering%
   \includegraphics[scale=3]{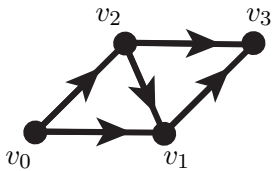}\\
      \rput(-1.7,0.4){$v_0$}
      \rput(0.4,0.4){$v_1$}
      \rput(-0.5,2.3){$v_2$}
      \rput(1.5,2.3){$v_3$}
   \caption{The rhombus graph $H$.}
   \label{fig: H}
   \end{figure}

The \emph{rhombus} directed graph $H$, is defined to be the graph satisfying
$V_H=\{v_0,v_1,v_2,v_3\}$ and $E_H=\{v_0v_1,v_0v_2,v_2v_3,v_1v_3,v_2v_1\}$.
We call $v_0$ the base vertex of $H$.
  \begin{figure}[htb!]
   \centering%
   \includegraphics[scale=3]{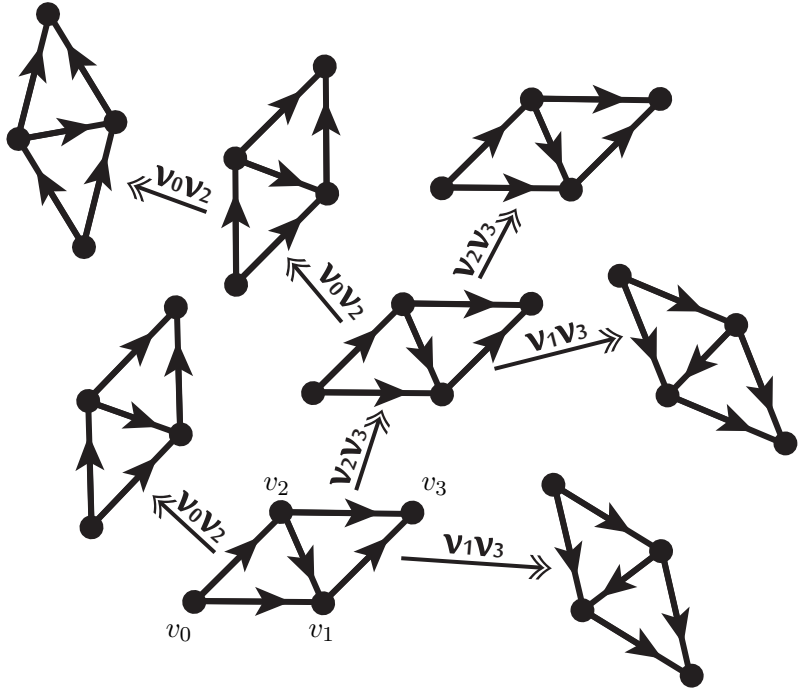}\\
      \rput(-3,1.3){$v_0$}
      \rput(-1.1,1.3){$v_1$}
      \rput(-1.7,3.25){$v_2$}
      \rput(0.4,3.25){$v_3$}
   \caption{A portion of $H^*$, including the names of the vertices of the node $H^*_{\text{root}}$ and the labels on the arcs.}
   \label{fig: H-Star}
   \end{figure}

We further define $H^*$ to be the infinite directed trinary tree whose nodes
are copies of $H$, labeling the three arcs emanating from every node by
$v_0v_2$, $v_2v_3$ and $v_1v_3$. We write $L(a)$ for the label of an arc $a$.
Let $N$ be a node of $H^*$, and let $v_iv_j\in E_H$; we call a pair $(N,v_i)$
a vertex of $H^*$, and a pair $(N,v_iv_j)$, an edge of $H^*$. Notice the
distinction between arcs of $H^*$ and edges of $H^*$, and the distinction
between nodes and vertices. The root of $H^*$ is denoted by
$H^*_{\text{root}}$. A portion of $H^*$ is depicted in figure~\ref{fig:
H-Star}.

There exists a natural map $\pi$ from the vertices of $H^*$ to the vertices
of $T^*$ which maps each node of $H^*$ to a pair of adjacent triangles of
$T^*$. $\pi$ is defined in such a way that $H^*_{\text{root}}$ is mapped to
${T^*_\text{root}}$ and to one of its neighboring triangles, and every
directed arc $MN$ of $H^*$, satisfies $\pi((M,L(MN)))=\pi((N,v_0v_1))$ (in
the sense of mapping origin to origin and destination to destination). In the
rest of the paper we extend $\pi$ naturally to edges and subgraphs, and
abridge $\pi((N,v))$ to $\pi(N,v)$. A portion of $T^*$ and its covering by
$H^*$ through $\pi$ are depicted in figure~\ref{fig: T star}.

  \begin{figure}[htb!]
   \centering%
   \includegraphics[scale=3]{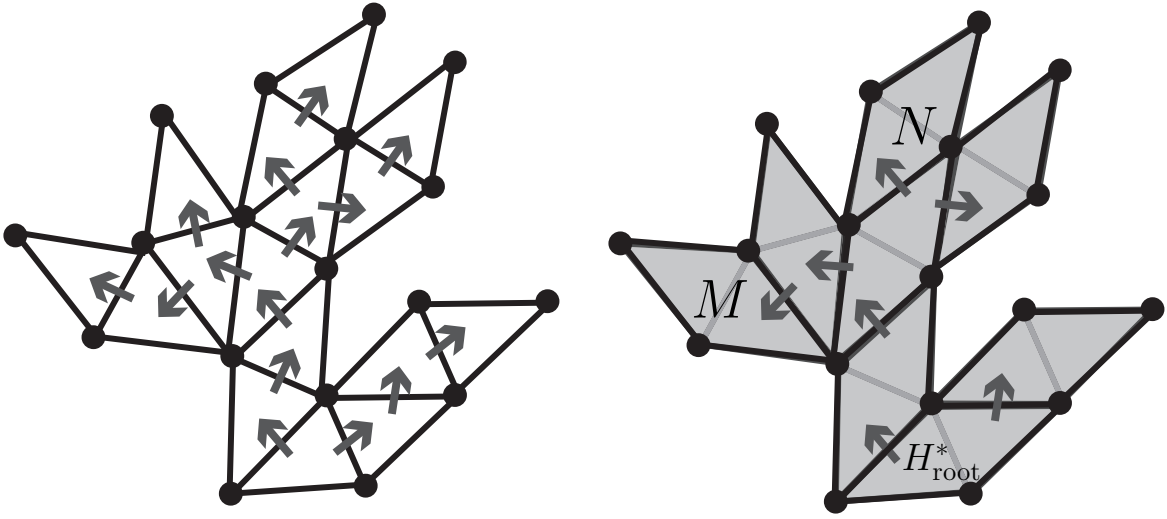}\\
      \rput(4.38,5.73){\huge $N$}
      \rput(1.8,3.43){\huge $M$}
      \rput(4.75,1.26){\Large $H^*_{\text{root}}$}
   \caption{A portion of $T^*$ and the corresponding covering by $H^*$. The orientation of the
   edges is omitted to simplify the drawing. The nodes $M$ and $N$ are QR-encoded by
   $S_M=(v_0v_2,v_2v_3,v_0v_2,v_0v_2)$, $\QR(M)=((0,1,0,0),(0,0,0),3)$ and  $S_N=(v_0v_2,v_2v_3,v_2v_3)$, $\QR(N)=((0,2),(0),1)$ respectively.}
   \label{fig: T star}
   \end{figure}

{\bf Encoding the rhombi.} In order to embed $T^*$ into $\C$,
rhombus-by-rhombus, a way to refer to every node $N\in H^*$ is called for. We
encode $N$ by the sequence of labels on the path from $H^*_{\text{root}}$ to
$N$. This trinary sequence is denoted by $S_N$. The map $N\to S_N$ is a bijection.

One may think of each label in $S_N$ as a direction, "left", "right" or
"forward", in which one must descend $H^*$, until finally arriving at $N$. To
simplify our proofs, we further encode $S_N$, by describing this sequence by
"how many forward steps to take between each turn left or right" and "is the
$i$-th turn left or right".

Formally, we do this by further encoding $S_N$ using a triple
$(\{q_i(N)\}_{i=1}^{m(N)+1},\{\rho_i(N)\}_{i=1}^{m(N)},m(N))$. We set
$q_i(N)$ to be the number of $v_2v_3$-s between the $(i-1)$-th non-$v_2v_3$
label in $S_N$ and the $i$-th one (for $i=1$ and for $i=m(N)+1$, the number
of $v_2v_3$-s before the first non-$v_2v_3$ label in $S_N$ and after the last
non-$v_2v_3$ label in $S_N$, respectively). We set $\rho_i(N)$ to be $0$ if
the $i$-th non-$v_2v_3$ element is $v_0v_2$ and $1$ if it is $v_1v_3$. We
call the triple $(\{q_i(N)\},\{\rho_i(N)\},m(N))$ the \emph{QR-encoding} of
$N$ denoting it by $\QR(N)$.

In accordance with our informal introduction, a QR-encoding
$(\{q_i\},\{\rho_i\},m)$, should be interpreted as taking $q_1$ steps
forward, then turning left or right according to $\rho_1$ being $0$ or $1$
respectively, then taking another $q_2$ steps forward in the new direction
and so on and so forth. The QR-encoding of each node is unique.

{\bf Encoding the vertices of $T^*$.} The encoding of the nodes of $H^*$
naturally extends to an encoding of the vertices of $T^*$ by defining
$\QR(u)=\{\QR(N)\ :\ \pi(N,v_0)=u\}$ for $u\in T^*$. This is indeed an
encoding of all the vertices of $T^*$, as for every vertex $u\in T^*$ there
exists at least one node $N$ such that $\pi(N,v_0)=u$. However, it is not
unique, as an infinite number of nodes encode each vertex.
As a unique encoding of every vertex is desirable
 for our purpose, we make the following observation.

\begin{obs}\label{obs: properties of rho-q}
Let $u\in T^*$, there exists a unique node $N$ such that
$\QR(N)=(\{q_i\},\{\rho_i\},m)\in\QR(u)$, satisfying $q_{m+1}=0$ and either
$q_{m}>0$ or $m=1$. We call such an encoding the proper encoding of $u$.
\end{obs}
\begin{proof}
It is not difficult to observe that the only proper encodings of
$\pi(H^*_{\text{root}},v_0)$ and $\pi(H^*_{\text{root}},v_1)$ are
$((0,0),(0),1)$ and $((0,0),(1),1)$ respectively.

For every vertex $u\in T^*$, except from $\pi(H^*_{\text{root}},v_0)$ and
$\pi(H^*_{\text{root}},v_1)$, there exists a unique node $N_u\in H^*$
satisfying that $\pi(N_u,v_i)=u$ for some $i\in\{2,3\}$. Let $\sim$ denote
the concatenation operation between sequences. Using this notation we have
that either $S(N_u)\sim v_2v_3\sim v_0v_2$ or $S(N_u)\sim v_2v_3\sim v_1v_3$
encode a node whose base vertex is mapped by $\pi$ to $u$. One may verify
from the definition of QR-encodings that $S_N$ ending with either $v_2v_3,
v_0v_2$ or with $v_2v_3\sim v_1v_3$ is equivalent to $q_{m+1}=0$ and
$q_{m}>0$.
\end{proof}

{\bf Polynomial embeddings.} A \emph{ $d$-polynomial embedding of a graph $G$
using $k$ edge-lengths} is a \underline{one-to-one} mapping $\psi:V_G\to
\C[x_1,\dots,x_d]$ where $\C[x_1,\dots,x_d]$ is the space of complex
polynomials in $d$ variables, such that for every fixed
$x\in\T^d=\{(x_1,\dots,x_d)\in\C^d:\forall i\in\{1,\dots,d\}, |x_i|=1\}$ the
map $v\mapsto \psi(v)(x)=\psi_x(v)$ is a linear embedding using only $k$
non-zero edge-lengths.

The importance of $d$-polynomial embeddings to our purpose stems from the
following proposition:
\begin{propos}\label{prop: almost every polynomial pre-drawing is degenerate}
If $\psi$ is a $d$-polynomial embedding of a graph $G$ with $k$ edge-lengths,
then for almost every $x=(x_1,\dots,x_d)\in \T^d$, $\psi_x$ is a degenerate
drawing of $G$ with $k$ edge-lengths.
\end{propos}
\begin{proof}
For any $v, w\in V_G$, the polynomials $\psi(v)(x)$ and $\psi(w)(x)$ may
coincide only on a set of measure $0$ in $T^d$. Taking union over all the
pairs $v_1,v_2$, we get that outside an exceptional set of measure zero in
$\T^d$, the map $\psi_x$ is one-to-one.
\end{proof}

\section{Three Distances Suffice for Degenerate Drawings}\label{sec: Two dist}

In this section we prove Proposition~\ref{prop: 3 degenerate} and thus
Theorem~\ref{thm: 3 degenerate}. To do so, for $x_0,x_1\in\T$, we introduce
in section~\ref{subs: psi} a $2$-polynomial embedding
$\psi=\psi(x_0,x_1)=\psi_{(x_0,x_1)}:T^*\to\C$. In section~\ref{subs: psi
img} we then write an explicit formula for the image of every vertex $v$
under $\psi$. This we do using the QR-encoding introduced in the
preliminaries section. In section~\ref{subs: psi poly embd} we prove that
$\psi$ is one-to-one. Finally, in section~\ref{subs: wrap up degenerate} we
conclude the proof of Proposition~\ref{prop: 3 degenerate}.

\subsection{The definition of $\psi$}\label{subs: psi} In this section we
define $\psi$. An outline of our construction is as follows: we start by
presenting $\psi_H(x)$, a $1$-polynomial embedding of $H$ which embeds the
rhombus graph onto a rhombus of side length $1$ with angle $x$ (identifying
the complex number $x$ with its angle on the unit circle). We then use a
boolean function $\ty$ on the nodes of $T^*$ to decide whether each rhombus
is mapped to a translated and rotated copy of $H(x_0)$ or of $H(x_1)$.
Finally, we define $\psi$ in the only way that respects both the covering
$\pi$ and the function $\ty$. The image of several subsets of $T^*$ through
$\psi(x_0,x_1)$ is depicted in figure~\ref{fig: psi T-star}.

We set $\psi_H(x)(v_0)=0$, $\psi_H(x)(v_1)=1$, $\psi_H(x)(v_2)=x$ and
$\psi_H(x)(v_3)=x+1$. This is indeed a polynomial drawing, mapping the
rhombus graph to a rhombus of edge length $1$, whose $v_1v_0v_2$ angle is
$x$. Figure~\ref{fig: phi} illustrates the image of $H$ under $\psi_H$.

  \begin{figure}[htb!]
   \centering%
   \includegraphics[scale=3]{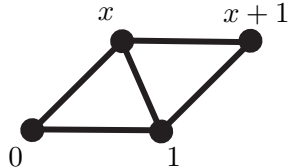}\\
      \rput(-1.7,0.4){$0$}
      \rput(0.4,0.4){$1$}
      \rput(-0.5,2.3){$x$}
      \rput(1.5,2.3){$x+1$}
   \caption{The image of $H$ under $\psi_H$. Observe how $x$ determines the $v_1v_0v_2$ angle of the rhombus.}
   \label{fig: phi}
   \end{figure}

We define an auxiliary function $\ty$. Let $MN$ be an arc of $H^*$. We set
\begin{equation}\label{eq: ty}
\ty(N)=\begin{cases} \ty(M) & L(MN)= v_2v_3\\
          \ty(N)\oplus q_{m(M)+1}(M)\pmod2 & L(MN) = v_0v_2\\
          \ty(N)\oplus q_{m(M)+1}(M)\oplus1\pmod2 & L(MN) = v_1v_3\\          \end{cases},
\end{equation}
where $\oplus$ represents addition modulo $2$. We set
$\ty(H^*_{\text{root}})=0$.

Set $\psi(\pi(H^*_{\text{root}}))=\psi_H(x_0)(H)$. Let $M,N\in H^*$ be
a pair of nodes such that $MN$ is an arc of $H^*$, and assume that $\psi$ is
already defined on the vertices of $\pi(M)$. By $\pi$'s definition, this
implies that $\psi(\pi(N,v_0))$ and $\psi(\pi(N,v_1))$ are already defined. We
then define $\psi(\pi(N,v_2)),\psi(\pi(N,v_3))$ so that
$\psi(\pi(N,v_0)),\psi(\pi(N,v_1)),\psi(\pi(N,v_2)),\psi(\pi(N,v_3))$ form a
translated and rotated copy of $H(x_{\ty(N)})$.
  \begin{figure}[htb!]
   \centering%
   \includegraphics[scale=3]{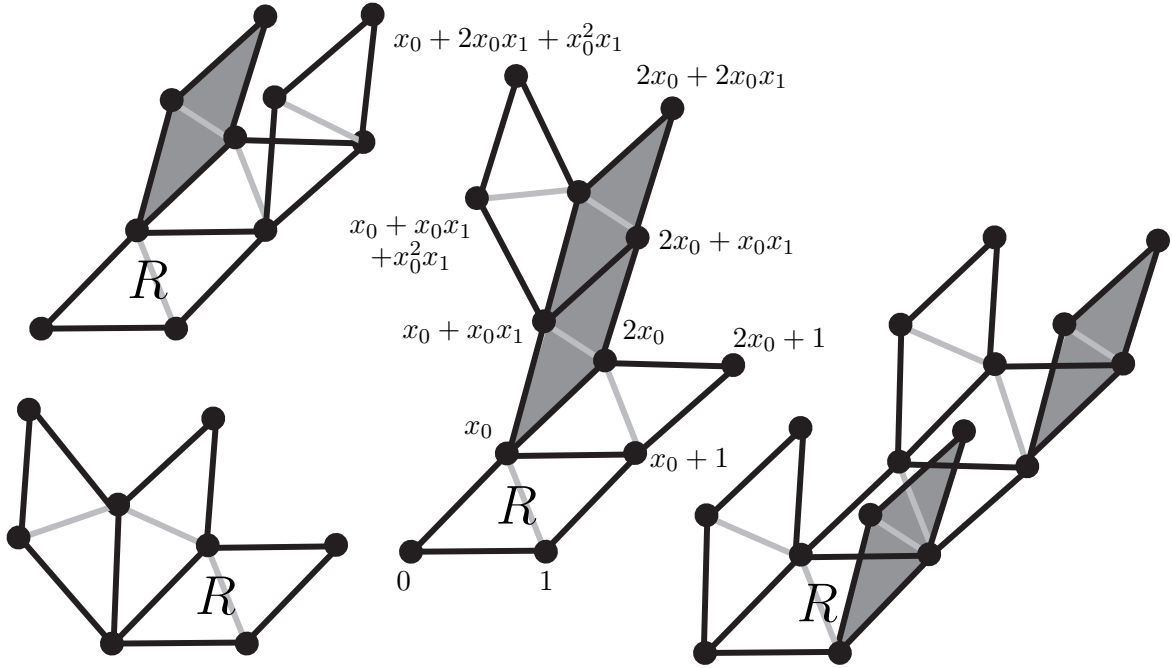}\\
      \rput(-2.5,1.7){$0$}
      \rput(-0.6,1.7){$1$}
      \rput(-1.5,3.7){$x_0$}
      \rput(1.3,3.3){$x_0+1$}
      \rput(2.5,4.9){$2x_0+1$}
      \rput(0.7,5){$2x_0$}
      \rput(1.8,6.2){$2x_0+x_0x_1$}
      \rput(-1.7,5){$x_0+x_0x_1$}
      \rput(-2.4,6.4){$x_0+x_0x_1$}
      \rput(-2.4,6){$+x_0^2x_1$}
      \rput(-1.1,8.9){$x_0+2x_0x_1+x_0^2x_1$}
      \rput(1.6,8.4){$2x_0+2x_0x_1$}
      \rput(3,1.4){\huge $R$}
      \rput(-5,1.5){\huge $R$}
      \rput(-1,2.7){\huge $R$}
      \rput(-5.9,5.7){\huge $R$}
   \caption{The image of several subgraphs of $T^*$ under $\psi$. Explicit values are given for several vertices.
   In each graph, the image of $\pi(H^*_{\text{root}})$ under $\psi$ is marked by $R$. Rhombi of angle $x_1$ are dark.}
   \label{fig: psi T-star}
   \end{figure}

As the image of every edge in $T^*$ is isometric to some edge of either
$H(x_0)$ or $H(x_1)$, we get
\begin{obs}\label{obs: psi three edge lengths}
Every edge of $T^*$ is mapped through $\psi$ to an interval of length $1$, $|x_0-1|$,
or $|x_1-1|$.
\end{obs}

While this definition of $\psi(x_0,x_1)$ is complete, an explicit formula for
every vertex in $T^*$ under $\psi(x_0,x_1)$ is required for proving that
$\psi$ is indeed a polynomial embedding. We devote the next section to
develop this formula.

\subsection{The image of $\psi$}\label{subs: psi img}
In this section we state a formula for $\psi\circ \pi$ of every base vertex.

 Let $u\in T^*$ and let $N\in H^*$, such
that $\QR(N)=(\{q_k\},\{\rho_k\},m)$ is the proper encoding of $u$. The first
$i$ elements of $\{q_k\},\{\rho_k\}$ encode a node in $T^*$ which is denoted
by $N_i$ (where $N_0=H^*_{\text{root}}$ which corresponds to the null
sequence). Naturally, $N_m=N$. From \eqref{eq: ty} we get
\begin{equation}\label{eq: nu-i}
\ty(N_i)=\ty(N_{i-1})\oplus q_{i}\oplus \rho_{i}\pmod2.
\end{equation}

Observe that in the embedding of every $H^*$ node through $\psi$, the edges
$v_0v_2$, $v_1v_3$ are parallel, as are the edges $v_0v_1, v_2v_3$. Next, we
define $P_i(x_0,x_1)$ to be a unit vector in the direction of the edges
$(v_0,v_1),(v_2,v_3)$ in $\psi(\pi(N_i))$ which, for $i>0$, is the same as
the direction of $(v_0,v_2),(v_1,v_3)$ in $\psi(\pi(N_{i-1}))$.

Formally
$$P^u_i(x_0,x_1)=P_i(x_0,x_1)=\psi(\pi(N_i,v_1))-\psi(\pi(N_i,v_0))=\psi(\pi(N_{i-1},v_2)-\psi(\pi(N_{i-1},v_0)),$$
where the last equality holds for $i>0$. Notice that $P_0(x_0,x_1)=1$.

With this in mind, it is possible to follow the change in $P_i$ between one
$N_i$ and the next. This yields:
\begin{equation}\label{eq: P formula}
P_{i}(x_0,x_1)=P_{i-1}(x_0,x_1)\cdot x_{\ty(N_{i-1})}.
\end{equation}

For $0\le i\le m$ write
$Q^u_i(x_0,x_1)=Q_i(x_0,x_1)=\psi_{x_0,x_1}(\pi(N_i,v_0))$. Observe that
$Q_0=\psi(\pi(H^*_\text{root},v_0))=0$. Let us describe how to get $Q_i$ from
$Q_{i-1}$ using $(\{q_k\},\{\rho_k\},m)$. By definition,
$$Q_i(x_0,x_1)-Q_{i-1}(x_0,x_1)=\psi(\pi(N_i,v_0))-\psi(\pi(N_{i-1},v_0)).$$
Thus $Q_i(x_0,x_1)-Q_{i-1}(x_0,x_1)$ can be calculated from the labels of the
edges along the path connecting $(N_{i-1},v_0)$ and $(N_i,v_0)$. Each edge
labeled $v_2v_3$ contributes to this difference $P_{i}$, and thus in total
such edge contribute $q_i\cdot P_{i}$. An edge with label $v_1v_3$
contributes $P_{i}/{x_{\ty(N_{i-1})}}=P_{i-1}$, while an edge labeled
$v_0v_2$ does not change the base vertex at all.

Applying this to the encoding, we get that
\begin{equation*}
Q_i-Q_{i-1} = q_{i}\cdot P_{i}+\rho_{i}\cdot P_{i-1}.
\end{equation*}

Summing this over $1 \le i\le m$, we get:
\begin{align*}
\psi_{x_0,x_1}(u)=Q_{m}&=\sum_{i=1}^m \left(  q_i P_i + \rho_{i}P_{i-1} \right) \\
&= \rho_1+\sum_{i=1}^{m-1} \left( q_i+\rho_{i+1}\right) P_i+q_{m}P_{m}
\end{align*}

Equivalently, letting $\rho_m=0$ and $q_0=0$ we have
\begin{equation}\label{eq: psi formula}
\psi_{x_0,x_1}(u)= \sum_{i=0}^m \left( q_i+\rho_{i+1}\right) P_i(x_0,x_1)= \sum_{i=0}^m  c_i P_i(x_0,x_1),
\end{equation}
where
\begin{equation}\label{eq: c-i}
c_i=q_i+\rho_{i+1}.
\end{equation}
Observe that for every $u\in T^*$, $\psi_{x_0,x_1}(u)$ is a polynomial in
$x_0$ and $x_1$ (because $P_i$ are monomials). Also observe that the total
degree of $P_i$, which we denote by $\degr P_i$, obeys $\degr P_i=\degr
P_{i-1}+1$. Therefore $\{c_i\}$ may be regarded as the coefficients of the
polynomial $\psi_{x_0,x_1}(u)$.

Note that in particular, using the above notations, Observation~\ref{obs:
properties of rho-q} and the fact that $(\{q_k\},\{\rho_k\},m)$ is proper yield
\begin{equation}\label{eq: c-m}
c_m=q_m>0.
\end{equation}

\subsection{Showing that $\psi$ is a
polynomial embedding}\label{subs: psi poly embd}

In this section we show that the image of the vertices of $T^*$ under $\psi$
are all distinct. Relation \eqref{eq: psi formula} and Observation~\ref{obs:
psi three edge lengths} imply that if this is the case, then $\psi$ is a
polynomial embedding of $T^*$ using three edge lengths.

The main proposition of this section is the following:
\begin{propos}\label{prop: psi is polynomial embedding}
Let $u,w\in T^*$ be two distinct vertices. Then $\psi_{x_0,x_1}(u)$ and
$\psi_{x_0,x_1}(w)$ are distinct polynomials.
\end{propos}
\begin{proof}
Let $(\{q^u_i\},\{\rho^u_i\},m)$,$(\{q^w_i\},\{\rho^w_i\},n)$ be the proper
QR-encoding sequences for $u,w$ respectfully, and let $N_k^u$ and $N_k^w$ be
the nodes encoded by the first $k$ elements of those sequences respectively.
We write $\nu^u_i=\ty(N_{i}^u)$, $\nu^w_i=\ty(N_{i}^w)$ for all $i$.

Notice that by Observation~\ref{obs: properties of rho-q} the two sequences
are distinct. The fact that $\pi(H^*_{\text{root}},v_0)$ and
$\pi(H^*_{\text{root}},v_1)$ have unique images under $\psi$ is
straightforward, as these are the only vertices whose image is a polynomial
of total degree $0$. We can therefore assume $m>1$.

Assume for the sake of obtaining a contradiction that
$\psi_{x_0,x_1}(u)\equiv\psi_{x_0,x_1}(w)$ as functions of $(x_0,x_1)$, and
thus in particular $c^u_i=c^w_i$ for all $i$.

Combining this with \eqref{eq: c-m} and Observation~\ref{obs: properties of
rho-q}, we get $m=n$.

Let $j$ be the first index to satisfy $(q^u_j,\rho^u_j) \neq
(q^w_j,\rho^w_j)$. By \eqref{eq: nu-i} and \eqref{eq: P formula} this implies
\begin{equation}\label{eq: P-j and nu-j}
\forall i\le j : P^u_{i}=P^w_{i} \text{ and } \nu^u_{i-1}=\nu^w_{i-1}.
\end{equation}
Moreover, since $q^u_{j-1}=q^w_{j-1}$ and $c^u_{j-1}=c^w_{j-1}$ we get by
\eqref{eq: c-i} that $\rho^u_{j}=\rho^w_{j}$. We deduce that $q^u_{j}\neq
q^w_{j}$. Since $c^u_{j}=c^w_{j}$ we have
$q^u_{j}-q^w_{j}=\rho^w_{j+1}-\rho^u_{j+1}\in\{-1,0,1\}$. As we have assumed
this difference to be non-zero, we may assume without loss of generality
\begin{align}\label{eq: dif q-j}
q^u_{j}-q^w_{j}=\rho^w_{j+1}-\rho^u_{j+1}=1.
\end{align}

Applying \eqref{eq: P-j and nu-j} for $i=j+1$ and the last relation to \eqref{eq: nu-i}, we get
$$\nu^u_{j}=\nu^u_{j-1}\oplus q^u_j\oplus\rho^u_j=\nu_{j-1}^w\oplus q^w_{j}\oplus\rho^w_j\oplus 1=\nu_{j}^w\oplus 1.$$

By \eqref{eq: P formula} we get
\begin{equation*}
\frac{P^u_{j+1}}{P^w_{j+1}}=\frac{x_{\nu^u_j}}{x_{\nu^w_{j}}}\neq 1,
\end{equation*}
which implies $c^u_{j+1}=c^w_{j+1}=0$. This in turn implies that $q^u_{j+1}=q^w_{j+1}=0$ and $\rho^u_{j+2}=\rho^w_{j+2}=0$. Using now relation \eqref{eq: nu-i} for $i=j+2$ and recalling \eqref{eq: dif q-j}, we get
$$ \nu^u_{j+2}=\nu^u_{j+1}\oplus 0\oplus\rho^u_{j+1}=
(\nu^w_{j+1}\oplus 1)\oplus 0\oplus(\rho^w_{j+1}-1)=\nu^w_{j+2}.$$
Again by \eqref{eq: P formula} we have
\begin{equation*}
\frac{P^u_{j+2}}{P^w_{j+2}}=\frac{P^u_{j+1}}{P^w_{j+1}}\cdot \frac{x_{\nu^u_{j+2}}}{x_{\nu^w_{j+2}}}=\frac{P^u_{j+1}}{P^w_{j+1}}\neq 1,
\end{equation*}
which implies $c^u_{j+2}=c^w_{j+2}=0$. Continuing by induction, we conclude that $c^u_{j+k}=c^w_{j+k}=0$ for all $k>1$. Thus $j=m$, and so by \eqref{eq: c-m},
$q^u_{j}=c^u_{j}=c^w_{j}=q^w_{j}$, a contradiction to \eqref{eq: dif q-j}.
\end{proof}
\subsection{Three Distances Suffice for Degenerate Drawings}\label{subs: wrap up degenerate}

We are now ready to present
the proof of Proposition~\ref{prop: 3 degenerate},
and thus conclude the proof of Theorem~\ref{thm: 3 degenerate}.
\begin{proof}[Proof of Proposition~\ref{prop: 3 degenerate}]
By Proposition~\ref{prop: psi is polynomial embedding}, $\psi$ is a
$2$-polynomial embedding of every finite subgraph $G\subseteq T^*$, using $3$
edge-lengths. By Proposition~\ref{prop: almost every polynomial pre-drawing
is degenerate} and Observation~\ref{obs: psi three edge lengths}, the set
$$\{(x_0,x_1)\ : \ \text{s.t. } x_0,x_1\in\T^2\text{ and } \psi(x_0,x_1) \text{ is a degenerate drawing}\}$$
is of full measure, and each of these degenerate drawings uses only side
lengths $1$, $|x_0-1|$ and $|x_1-1|$. Let $a\in(0,1)$, the embedding
$a\cdot\psi$, i.e. the composition of a multiplication by $a$ on $\psi$, is
thus a degenerate drawing of $G$ for almost every $x_0, x_1$ using the side
lengths $a, a|1-x_0|,a|1-x_1|$. The desired result follows.
\end{proof}

\subsection{Open problems}
Several interesting problems concerning graphs with a low
(degenerate) distance
number remain open. In this short section we state those of greater interest
to us. The first and most natural one is:

\begin{problem}
Do outerplanar graphs have a uniformly bounded distance number?
\end{problem}

While we believe we may be able to answer this problem positively, our
construction is rather complicated and is thus postponed to a future paper.
It will be interesting to see a simple construction which can be easily
described.

The general problem which, in our opinion,
extends this work most naturally is:
\begin{problem}
Which families of graphs have a uniformly bounded (degenerate) distance
number?
\end{problem}
Observe that the family of planar graphs does not have this
property, as the
complete bipartite graph $K_{2,n}$ is an example of a planar graph whose
degenerate distance number is $\Theta(\sqrt{n})$.

Finally, our result implies that the maximum possible
degenerate distance number of an
outerplanar graph is at most three. It is easy to see that there are
outerplanar graphs whose degenerate distance number is two. Are there any
outerplanar graphs whose degenerate distance number is indeed three?

\begin{problem}
Is it true that the maximum possible degenerate distance number of an
outerplanar graph is two?
\end{problem}

\end{document}